\documentclass[12pt]{article}

\usepackage{amsmath,amssymb,amscd,amsfonts,amsthm,slashbox}
\usepackage{hyperref}
\usepackage{float}
\usepackage{fullpage}
\DeclareMathOperator{\diag}{diag}

\title{Hankel Matrices for the Period-Doubling Sequence}

\author{Robbert J. Fokkink and Cor Kraaikamp \\
Applied Mathematics \\
TU Delft\\
Mekelweg 4\\
2628 CD Delft\\
Netherlands \\
\href{mailto:r.j.fokkink@ewi.tudelft.nl}{\tt r.j.fokkink@ewi.tudelft.nl} \\
\href{mailto:C.Kraaikamp@tudelft.nl}{\tt C.Kraaikamp@tudelft.nl}
\and
Jeffrey Shallit \\
School of Computer Science\\
University of Waterloo\\
Waterloo, ON  N2L 3G1 \\
Canada \\
\href{mailto:shallit@cs.uwaterloo.ca}{\tt shallit@cs.uwaterloo.ca}
}

\begin{document}

\theoremstyle{plain}
\newtheorem{theorem}{Theorem}
\newtheorem{corollary}[theorem]{Corollary}
\newtheorem{lemma}[theorem]{Lemma}
\newtheorem{proposition}[theorem]{Proposition}

\theoremstyle{definition}
\newtheorem{definition}[theorem]{Definition}
\newtheorem{example}[theorem]{Example}
\newtheorem{conjecture}[theorem]{Conjecture}

\theoremstyle{remark}
\newtheorem{remark}[theorem]{Remark}

\maketitle

\begin{abstract}
We give an explicit evaluation, in terms of products of Jacobsthal
numbers, of the Hankel determinants of order
a power of two for the period-doubling sequence.  We also explicitly
give the eigenvalues and eigenvectors of the corresponding Hankel
matrices.     Similar considerations give the Hankel determinants
for other orders.
\end{abstract}

\section{Introduction}

Let ${\bf s} = (s_n)_{n \geq 0}$ be a sequence of real numbers.
The {\it Hankel matrix} $M_{\bf s} (k)$ of order $k$ associated with
$\bf s$ is defined as follows:
\begin{equation}
M_{\bf s} (k) = \left[
	\begin{array}{cccc}
	s_0 & s_1 & \cdots & s_{k-1} \\
	s_1 & s_2 & \cdots & s_k \\
	\vdots & \vdots & \ddots & \vdots \\
	s_{k-1} & s_k & \cdots & s_{2k-2}
	\end{array} \right] .
	\label{hank}
\end{equation}
See, for example, \cite{Iohvidov:1982}.
Note that the rows of $M_{\bf s} (k)$ are made up of successive length-$k$
``windows'' into the sequence $\bf s$.

Of particular interest are the determinants
$\Delta_{\bf s} (k) = \det M_{\bf s} (k)$ of the Hankel matrices
in \eqref{hank}, which are often quite challenging to compute explicitly.
In some cases when these determinants are non-zero,
they permit estimation of the
irrationality measure of the associated real numbers
$\sum_{n \geq 0} s_n b^{-n}$, where $b \geq 2$ is an integer;
see, for example,
\cite{Bugeaud:2011,Coons&Vrbik:2012,Coons:2013,Coons:2015,Vaananen:2015,Niu&Li:2015,Bugeaud&Han&Wen&Yao:2015}.
In some sense, the Hankel determinants
measure how ``far away'' the sequence $\bf s$ is from a linear recurrence
with constant coefficients, since for such a sequence we have
$H_{\bf s}(n) = 0$
for all sufficiently large $n$.

In this note we consider the Hankel determinants for a
certain infinite sequence of interest, the so-called
{\it period-doubling sequence} 
$${\bf d} = (d_i)_{i \geq 0} = 1011101010111011101110101011101 \cdots .$$
This sequence can be defined in various ways \cite{Damanik:2000},
but probably the three
simplest are as follows:
\begin{itemize}
\item as the fixed point of the map
$$ 1 \rightarrow 10, \quad\quad\quad 0 \rightarrow 11;$$
\item as the first difference, taken modulo $2$, of the Thue-Morse
sequence ${\bf t} = 0110100110010110 \cdots$ (fixed point of the map
$0 \rightarrow 01$, $1 \rightarrow 10$);
\item as the sequence defined by
$$ d_i = \begin{cases}
	1, & \text{if $s_2 (i) \not\equiv s_2 (i+1)$ (mod $2$)}; \\
	0, & \text{otherwise};
	\end{cases}
$$
where $s_2 (i)$ is the sum of the binary digits of $i$ when expressed
in base $2$.
\end{itemize}
We explicitly compute the Hankel determinants when the orders
are a power of $2$, and we also compute the eigenvalues and eigenvectors of the
corresponding Hankel matrices.
We derive recursions for Hankel determinants for all orders.
Finally,
we also consider the determinants
for the complementary sequence 
$$\overline{\bf d} = 0100010101000100010001010100010 \cdots ,$$
obtained from $\bf d$ by changing $1$ to $0$ and vice versa.

\subsection{Previous work}

By considering $\Delta_{\bf d} (n)$ modulo $2$,
Allouche, Peyri\`ere, Wen, and
Wen \cite[Prop.~2.2]{Allouche&Peyriere&Wen&Wen:1998}
proved that $\Delta_{\bf d} (n)$ is odd for all $n \geq 1$.    However,
they did not obtain any explicit formula for $\Delta_{\bf d} (n)$.
In fact, their main focus was on the non-vanishing of the Hankel
determinants for the Thue-Morse sequence on values $\pm 1$.  For this,
also see Bugeaud and Han \cite{Bugeaud&Han:2014} and
Han \cite{Han:2015}.  Recently
Fu and Han \cite{Fu&Han:2015} also studied some Hankel matrices
associated with the period-doubling sequence, but they did not obtain
our result.

There are only a small number of sequences defined by iterated morphisms
for which the Hankel determinants are explicitly known (even for
subsequences).  These include the infinite Fibonacci word
\cite{Kamae&Tamura&Wen:1999}, the
paperfolding sequence \cite{Guo&Wen&Wu:2014,Fu&Han:2015},
the iterated differences of the Thue-Morse sequence
\cite{Guo&Wen:2014}, the Cantor sequence
\cite{Wen&Wu:2014}, and sequences related to the
Thue-Morse sequence \cite{Han&Wu:2015,Fu&Han:2016a,Fu&Han:2016b}.

\section{Hankel determinants}

Here are the first few terms of the Hankel determinants for the
period-doubling sequence and its complementary sequence:

\begin{table}[H]
\begin{center}
\begin{tabular}{c|rrrrrrrrrrrrrrrr}
$k$ & 1 & 2 & 3 & 4 & 5 & 6 & 7 & 8 & 9 & 10 & 11 & 12 & 13 & 14 & 15 & 16 \\
\hline
$\Delta_{\bf d}(k)$ & 1 & 1 & $-1$&$-3$ & 1 & 1 & $-1$&$-15$& 1 &  1 & $-1$ & $-3$ &  1 &  1 & $-9$ & $-495$\\
\hline
$\Delta_{\overline{\bf d}}(k)$ & 0 & $-1$ &$0$ & 1 & 0 & $-1$ & 0 &$9$&$0$&$-1$& 0 &  1 & $0$ & $-1$ &  0 & $225$
\end{tabular}
\end{center}
\end{table}

The large values at the powers of $2$ suggest something interesting
is going on.  Indeed, by explicit calculation we find
\begin{align*}
\Delta_{\bf d}(32) &= -467775 & \quad
	\Delta_{\overline{\bf d}} (32) &= 245025 \\
\Delta_{\bf d} (64) &= -448046589375 & \quad
	\Delta_{\overline{\bf d}} (64) &= 218813450625 \\
\Delta_{\bf d} (128) &= -396822986774382287109375 & \quad
	\Delta_{\overline{\bf d}} (128) &= 200745746250569862890625,
\end{align*}
and so forth. Another obvious pattern is $\Delta_{\overline{\bf d}}(n)=0$ for odd $n$.

Define $J_n = (2^n - (-1)^n)/3$, the so-called
{\it Jacobsthal numbers} \cite{Horadam:1996}.
It is easy to see that
\begin{align}
J_{n+1} &= J_n + 2J_{n-1} \label{jac1} \\
J_{n+1} &= 2J_n + (-1)^n \label{jac2} \\
J_{n+1} &= 2^n - J_{n} \label{jac3}
\end{align}
for $n \geq 0$. We will prove that $\Delta_{\bf d}(n)$ and $\Delta_{\overline{\bf d}}(n)$ are products
of $n$ Jacobsthal numbers, and that their factorizations are almost the same. The
reason why $\Delta_{\overline{\bf d}}(n)=0$ for odd $n$ is that it is a
product involving $J_0$.

In this paper we will prove

\begin{theorem}
For integers $k \geq 2$ we have
$\Delta_{\bf d} (2^k) = - J_{k+1} \prod_{3 \leq i \leq k} J_i^{2^{k-i}},$
and
$\Delta_{\overline{\bf d}} (2^k) = J_k \prod_{3 \leq i \leq k} 
J_i^{2^{k-i}} $, where, as usual, the empty product evaluates to $1$.
\label{main}
\end{theorem}

In the proof of Theorem~\ref{main}, we also obtain a complete description of the eigenvalues of
$M_{\bf d}(2^k)$ and $M_{\overline{\bf d}}(2^k)$, as well as a basis for the corresponding eigenspaces.

Our second main result handles the Hankel determinants of all orders.

\begin{theorem}
For all integers $n \geq 1$, the Hankel determinants
$\Delta_{\bf d} (n)$ and $\Delta_{\overline{\bf d}}(n)$ are products of $n$ Jacobsthal numbers,
counted with multiplicity, and including the trivial divisors $J_0, J_1, J_2$ in the count. If $n$ is even,
then $J_{i}\Delta_{\bf d}(n)=-J_{i+1}\Delta_{\overline{\bf d}}(n)$ for some $i>0$.
\label{main2}
\end{theorem}

\section{1-D and 2-D morphisms}

Let $\Sigma, \Delta$ denote finite alphabets.
A {\it morphism} (or {\it substitution}) is a map $h$ from
$\Sigma^* \rightarrow \Delta^*$ satisfying $h(xy) = h(x)h(y)$ for all
strings $x, y$.  If $\Sigma = \Delta$ we can iterate $h$, writing
$h^1(x)$ for $h(x)$, $h^2(x)$ for $h(h(x))$, and so forth.
In this paper we will need a variant of the
so-called Thue-Morse morphism \cite{Seebold:1985a}, defined as
follows:
$$ \rho(1) = (-1, 1) \quad \quad \quad \rho(-1) = (1, -1) .$$

We can also define the notion of morphisms for arrays (or matrices).
A {\it 2-D morphism} (or {\it 2-D substitution}) can be viewed
as a map from $\Sigma$ to $\Delta^{r \times s}$ that is extended
to matrices in the obvious way \cite{Salon:1986,Salon:1987,Salon:1989b,Shallit&Stolfi:1989}.

One of the most famous maps of this form is the map
$$ \gamma(1) = \left[ \begin{array}{cc}
	1 & 1 \\
	1 & -1
	\end{array}
	\right] \quad \quad
	\gamma(-1) = \left[ \begin{array}{cc}
	        -1 & -1 \\
		-1 & 1
		\end{array}
		\right],
$$
which, when iterated $k$ times, produces a {\it Hadamard matrix} of
order $2^k$.  (An $n \times n$ matrix $H$ is said to be Hadamard if
all entries are $\pm 1$ and furthermore $H H^T = n I$, where
$I$ is the identity matrix; see \cite{Seberry&Yamada:1992}.)

We now observe that the Hankel matrix $M_{\bf d} (2^k)$
of the period-doubling sequence can be generated in a similar
way, via the 2-D morphism

$$ \varphi(1) = \left[ \begin{array}{cc}
	1 & 0 \\
	0 & 1
	\end{array}
	\right] \quad \quad
	\varphi(0) = \left[ \begin{array}{cc}
	        1 & 1 \\
		1 & 1
		\end{array}
		\right],
$$
More precisely, $M_{\bf d} (2^k) = \varphi^k (1)$.

Similarly, $M_{\overline{\bf d}} (2^k) = \overline{\varphi}^k (0)$
for the complementary substitution $\overline\varphi$ which is defined
as follows:
$$ \overline\varphi(0) = \left[ \begin{array}{cc}
	0 & 1 \\
	1 & 0
	\end{array}
	\right] \quad \quad
	\overline\varphi(1) = \left[ \begin{array}{cc}
	        0 & 0 \\
		0 & 0
		\end{array}
		\right].
$$

Let $v = (a_1, a_2, \ldots, a_n)$ be a vector of length $n$.  By $\diag(v)$ we mean the diagonal
matrix
$$ \left[ \begin{array}{ccccc}
	a_1 & 0 & 0 & \cdots & 0 \\
	0 & a_2 & 0 & \cdots & 0 \\
	0 &  0  & a_3 & \cdots & 0 \\
	\vdots & \vdots & \vdots & \ddots & \vdots \\
	0 & 0 & 0 & \cdots & a_n
	\end{array} \right] .$$

We now observe that the Hankel matrices of the period-doubling sequence
are diagonalized by the Hadamard matrices $\gamma^k (1)$:

\begin{theorem}
For $k \geq 1$ we have
\begin{itemize}
\item[(a)]
$\gamma^k(1) \varphi^k (1) \gamma^k (1) = 2^k \diag(J_{k+1}, J_k, J_{k-1} \rho(1), J_{k-2} \rho^2 (1), \ldots, J_1 \rho^{k-1} (1) ) $
and
\item[(b)] $\gamma^k(1) \varphi^k (0) \gamma^k (1) = 2^{k+1} \diag(J_k, J_{k-1},
J_{k-2} \rho (1), \ldots, J_1 \rho^{k-2} (1), J_0\rho^{k-1} (1) ). $
\end{itemize}
\end{theorem}

\begin{proof}
By induction on $k$.  The verification for $k = 1$ is left to the reader.

Now assume the results are true for $k$.  We prove them for $k+1$.

We start with (a).
Write $P_k$ for the vector
$[J_{k+1}, J_k, J_{k-1} \rho(1), J_{k-2} \rho^2 (1), \ldots, J_1 \rho^{k-1} (1) ]$
and $Q_k$ for the vector
$[J_k, J_{k-1}, J_{k-2} \rho(1), \ldots, J_1 \rho^{k-2} (1), J_0
\rho^{k-1} (1)] $.  Note that from the definition of
$P_k$ and $Q_k$, and the
fact that $J_0 = 0$, we have
\begin{equation}
Q_{k+1} =  [P_k, \overbrace{0,0,\ldots,0}^{2^k} \ ] . \label{qp}
\end{equation}

Now
\begin{align*}
\gamma^{k+1} (1) \varphi^{k+1} (1) \gamma^{k+1} (1) &=
	\left[ \begin{array}{cc}
		\gamma^k (1) & \gamma^k (1) \\
		\gamma^k (1) & -\gamma^k (1)
		\end{array}
		\right]
	\left[ \begin{array}{cc}
		\varphi^k (1) & \varphi^k (0) \\
		\varphi^k (0) & \varphi^k (1)
		\end{array}
		\right]
	\left[ \begin{array}{cc}
		\gamma^k (1) & \gamma^k (1) \\
		\gamma^k (1) & -\gamma^k (1)
		\end{array}
		\right]
\\[1.0ex]
&= \left[ \begin{array}{cc}
	\gamma^k (1)(\varphi^k(1)+\varphi^k(0)) & \gamma^k (1)(\varphi^k(1)+\varphi^k(0))  \\
	\gamma^k (1)(\varphi^k(1) - \varphi^k(0)) & \gamma^k (1)(\varphi^k(0)-\varphi^k(1))
	\end{array}
	\right]
\left[ \begin{array}{cc}
       \gamma^k (1) & \gamma^k (1) \\
	\gamma^k (1) & -\gamma^k (1)
	\end{array}
	\right] \\[1.0ex]
&= \left[ \begin{array}{cc}
       2 \gamma^k (1)(\varphi^k (1) +\varphi^k(0)) \gamma^k(1) & {\bf 0} \\
	{\bf 0} & 2 \gamma^k (1)(\varphi^k (1) -\varphi^k(0)) \gamma^k(1)
	\end{array}
	\right] \\[1.0ex]
&= \left[ \begin{array}{cc}
	 2^{k+1} \diag(P_k + 2Q_k) & {\bf 0} \\
	{\bf 0} & 2^{k+1} \diag(P_k - 2Q_k)
	\end{array}
	\right],
\end{align*}
where by ${\bf 0}$ we mean the appropriately-sized matrix of all $0$'s.

Now, from \eqref{jac1} and \eqref{jac2} we see that
$[P_k + 2 Q_k, P_k - 2 Q_k] = P_{k+1}$, so the proof of the first claim
is complete.

Now let's verify (b):
\begin{align*}
\gamma^{k+1} (1) \varphi^{k+1} (0) \gamma^{k+1} (1) &=
	\left[ \begin{array}{cc}
		\gamma^k (1) & \gamma^k (1) \\
		\gamma^k (1) & -\gamma^k (1)
		\end{array}
		\right]
	\left[ \begin{array}{cc}
		\varphi^k (1) & \varphi^k (1) \\
		\varphi^k (1) & \varphi^k (1)
		\end{array}
		\right]
	\left[ \begin{array}{cc}
		\gamma^k (1) & \gamma^k (1) \\
		\gamma^k (1) & -\gamma^k (1)
		\end{array}
		\right]
\\[1.0ex]
&= \left[ \begin{array}{cc}
	2 \gamma^k (1)\varphi^k(1) & 2 \gamma^k(1) \varphi^k(1)  \\
	{\bf 0} & {\bf 0}
	\end{array}
	\right]
\left[ \begin{array}{cc}
       \gamma^k (1) & \gamma^k (1) \\
	\gamma^k (1) & -\gamma^k (1)
	\end{array}
	\right] \\[1.0ex]
&= \left[ \begin{array}{cc}
       4 \gamma^k (1)\varphi^k (1) \gamma^k(1) & {\bf 0} \\
	{\bf 0} &  {\bf 0}
	\end{array}
	\right] \\[1.0ex]
&= 4 \left[ \begin{array}{cc}
	 2^k \diag(P_k) & {\bf 0} \\
	{\bf 0} & {\bf 0}
	\end{array}
	\right] \\[1.0ex]
&= 2^{k+2} \diag(Q_{k+1})  , \\
\end{align*}
where we have used \eqref{qp}.
This completes the proof of (b).
\end{proof}

\begin{corollary}
The eigenvalues of $M_{\bf d} (2^k) $, with their multiplicities,
are as follows:
\begin{itemize}
\item $J_{k+1}$ with multiplicity $1$
\item $J_k$ with multiplicity $1$
\item $J_{k-i}$ and $-J_{k-i}$, each with multiplicity
	$2^{i-1}$, for $1 \leq i \leq k-3$.
\item $1$ and $-1$, each with multiplicity $3 \cdot 2^{k-3}$.
\end{itemize}

Furthermore, the basis for the eigenspace of each eigenvalue
can be read off from the respective columns of the Hadamard
matrix $H(2^k)$.
\end{corollary}

\begin{proof}
This follows immediately from the fact that
$$M_{\bf d}(2^{k+1}) = H(2^k) M_{\bf d}(2^k) H(2^k) =
2^k \diag(P_k) ,$$
and $H(2^k) = H(2^k)^T$, and
$H(2^k) H(2^k)^T = 2^k I$.
\end{proof}

\begin{corollary}
The eigenvalues of $M_{\overline{\bf d}} (2^k) $, with their multiplicities,
are almost the same:

\begin{itemize}
\item $J_k$ with multiplicity $1$
\item $-J_k$ with multiplicity $1$
\item $J_{k-i}$ and $-J_{k-i}$, each with multiplicity
	$2^{i-1}$, for $1 \leq i \leq k-3$.
\item $1$ and $-1$, each with multiplicity $3 \cdot 2^{k-3}$.
\end{itemize}
Again, the basis for the eigenspace of each eigenvalue
can be read off from the respective columns of the Hadamard
matrix $H(2^k)$.
\end{corollary}

\begin{proof}
This follows immediately from the fact that
$M_{\overline{\bf d}}(2^k)=E_k-M_{\bf d}(2^{k})$, for the
matrix $E_k$ that has all entries equal to $1$, and the
fact that
$\gamma^k(1) E_k \gamma^k (1) = \diag(2^{2k}, 0, 0, \ldots, 0 ) $.

\end{proof}

Finally, we get the proof of Theorem~\ref{main}:

\begin{proof}
The product of the eigenvalues of $M_{\bf d}(2^k)$ is
$$ - J_{k+1} \prod_{3 \leq i \leq k} J_i^{2^{k-i}},$$
and the product of the eigenvalues of $M_{\overline{\bf d}}(2^k)$ is
$$ J_k \prod_{3 \leq i \leq k} J_i^{2^{k-i}} .$$
\end{proof}

\section{General orders}

The Hankel determinants $\Delta_{\bf d}(2^k)$ and $\Delta_{\overline
{\bf d}}(2^k)$ are products of Jacobsthal numbers that correspond to
eigenvalues of their associated Hankel matrices.  For general $n$, the
Hankel determinants $\Delta_{\bf d}(n)$ are also products of Jacobsthal
numbers (as we will prove below), but these numbers no longer correspond
to eigenvalues of the Hankel matrix $M_{\bf d}(n)$. The Hankel
determinants $\Delta_{\overline {\bf d}}(n)$ are equal to zero if $n$
is odd.

\subsection{Preliminary observations}

An inspection of the Hankel determinants quickly reveals recursive
formulas, such as:

\begin{proposition}
For $k \geq 1$ we have $\Delta_{\bf d} (3 \cdot 2^k) =
\Delta_{\bf d}(2^k)$ and $\Delta_{\overline{\bf d}} (3 \cdot 2^k) =
\Delta_{\overline{\bf d}}(2^k)$.
\end{proposition}

\begin{proof}
We consider $\Delta_{\bf d}(3\cdot 2^k)$ first.
The result is easy to check for $k = 1$.  For $k \geq 2$,
the corresponding Hankel matrix is easily seen to be
$$ \left[ \begin{array}{ccc}
P & Q & P \\
Q & P & P \\
P & P & P
\end{array}
\right] $$
where $P = \varphi^n(1)$, $Q = \varphi^n (0)$.

Using Gaussian elimination, we can subtract the third row from each
of the first two rows,
obtaining
$$ \left[ \begin{array}{ccc}
0 & R & 0 \\
R & 0 & 0 \\
P & P & P
\end{array}
\right] $$
Now an easy induction gives that $R$ is an anti-diagonal matrix of
all $(-1)^k$'s, so for $k \geq 2$ we have $\det R = 1$.
We conclude that the determinant is indeed $\det P$.

If the same computation is carried out for $\Delta_{\overline{\bf d}}(3\cdot 2^k)$,
then we arrive at
$$ \left[ \begin{array}{ccc}
0 & -R & 0 \\
-R & 0 & 0 \\
\overline P & \overline P & \overline P
\end{array}
\right] $$
where $\overline P$ is a complementary matrix and $-R$ is an anti-diagonal matrix of
all $(-1)^{k+1}$'s. We conclude that the determinant is indeed $\det \overline P$.
\end{proof}

Similar computations give

\begin{proposition}
\ \vphantom{a}
\begin{itemize}
\item[(a)]  $\Delta_{\bf d} (5 \cdot 2^k) = \Delta_{\bf d}(2^k)$ and
$\Delta_{\overline {\bf d}} (5 \cdot 2^k) = \Delta_{\overline{\bf d}}(2^k)$ for
$k \geq 0$.

\item[(b)] $\Delta_{\bf d} (7 \cdot 2^k) = \Delta_{\bf d}(2^k)$
and $\Delta_{\overline{\bf d}} (7 \cdot 2^k) = \Delta_{\overline{\bf d}}(2^k)$
for
$k \geq 1$.

\item[(c)] $\Delta_{\bf d} (2^k - 1) =
- \prod_{3 \leq i \leq k-1} J_i^{2^{k-i}} $ for $k \geq 3$.
\end{itemize}
\end{proposition}

\subsection{Two recursions}

We derive two recursions to compute $\Delta_{\bf d}(n)$ and
$\Delta_{\overline{\bf d}}(n)$.  The derivation is the same for both
determinants. We restrict our attention to the first determinant, and
leave it to the reader to verify the recursion for the second
determinant.  If the second significant digit of the binary expansion
of $n$ is one, then we apply the first recursion.  If it is zero, then
we apply the second recursion.  Each recursion produces a power of a
Jacobsthal number and reduces $\Delta_{\bf d}(n)$ to $\Delta_{\bf
d}(n')$.  If $n$ has binary expansion of length $k$, then the binary
expansion of $n'$ is $k-1$. The recursion also produces a power of $2$,
which may be positive or negative, but since we know by
\cite[Prop.~2.2]{Allouche&Peyriere&Wen&Wen:1998}
that our Hankel determinants are odd
(for $\bf d$, not for $\overline {\bf d}$), we can ignore these
powers.

\subsubsection{Recursion one}
The Hankel matrix $M_{\bf d}(n)$ is an $n\times n$ submatrix in the larger Hankel matrix $M_{\bf d}(m)$ for any $n\leq m$.
We introduce some more notation. We write $P_k$ for $M_{\bf d}(2^k)$,
and $Q_k$ for $P_k-(-1)^kD_k$, where $D_k$ is the $2^k\times 2^k$ anti-diagonal matrix with all ones on the diagonal.
Our recursion involves the matrix $M_{i,k}(j)$,
which is the $j\times j$ submatrix of $J_iP_k+J_{i-1}Q_k$
consisting of the first $j$ columns and the first $j$ rows, where as before
$J_i$ is the $i$-th Jacobsthal number.
We denote the determinant of $M_{i,k}(j)$ by $\Delta_{i,k}(j)$.
If $i=1$,
then $M_{i,k}(j)$ is equal to the period doubling Hankel matrix
$M_{\bf d}(j)$, and
its determinant $\Delta_{i,k}(j)$ is equal to $\Delta_{\bf d}(j)$.
If $j\leq 2^{k-1}$, then the $j\times j$ blocks in $P_k$ and $Q_k$ coincide,
$M_{i,k}(j)$ is equal to $2^{i-1}M_{\bf d}(j)$,
and its determinant is equal to $2^{(i-1)j}\Delta_{\bf d}(j)$.
So the only interesting values are $2^{k-1}< j\leq 2^k$,
and we will only consider such $j$.

\begin{lemma}[Recursion one] If $2^{k}+2^{k-1}< j\leq  2^{k+1}$ then
\[\Delta_{i,k+1}(j)=\epsilon_k\cdot \frac{J_i^{2^{k}}}{2^{2^{k+1}-j}}\cdot\Delta_{i+1,k}(j-2^{k})\]
where 
$$\epsilon_k = \begin{cases}
	1,  & \text{if $k > 1$}; \\
	-1, & \text{if $k = 1$}.
	\end{cases}
$$
\end{lemma}

Observe that the recursion reduces $j$ by $2^{k}$ which is equal to the power
of the Jacobsthal number that is produced by the recursion.

\begin{proof} By definition $H_{i,k+1}(j)$ is the $j\times j$ block in
the matrix $J_iP_{k+1}+J_{i-1}Q_{k+1}$, which is equal to
\[
\left[
\begin{matrix}
J_i P_{k}+J_{i-1}P_{k}&J_i Q_{k}+J_{i-1}P_{k}\\
J_i Q_{k}+J_{i-1}P_{k}&J_i P_{k}+J_{i-1}P_{k}
\end{matrix}
\right]
=
\left[
\begin{matrix}
2^{i-1}P_{k}&J_i Q_{k}+J_{i-1}P_{k}\\
J_i Q_{k}+J_{i-1}P_{k}&2^{i-1}P_{k}
\end{matrix}
\right]
\]
Abbreviating this expression, we write this matrix as
\[
\left[
\begin{matrix}
A&B\\
B&A
\end{matrix}
\right]
\]
Perform Gaussian elimination by subtracting ``row'' 1 of this $2\times 2$
block matrix from row 2, and then subtract ``column'' 1 from column 2 to get
\[
\left[
\begin{matrix}
A&A+B\\
B-A&0
\end{matrix}
\right].
\]
The lower left block $B-A$ is an anti-diagonal matrix $(-1)^{k+1}J_iD$.
The upper right block $A+B$ is equal to
$(J_{i}+2J_{i-1})P_{k}+J_iQ_{k}=J_{i+1}P_{k}+J_iQ_{k}$, i.e., it is
is equal to $M_{i+1,k}(2^{k-1})$.
We started out with the $j\times j$ submatrix in the entire matrix.
The recursion essentially reduces it to the $(j-2^k)\times(j-2^k)$ submatrix in the upper right block $A+B$ by
getting rid of the first column
$\left[\begin{matrix}A\\B-A\end{matrix}\right]$, 
as follows:

The $j\times j$ submatrix extends over the $j-2^k$ top rows of the lower block $B-A$.
The first $2^{k+1}-j$ columns of this $(j-2^k)\times 2k$ submatrix are zero
and the last $j-2^k$ columns form the anti-diagonal $(-1)^{k+1}J_iD_{j-2^k}$.
Ignoring the sign of the determinant for the moment,
the submatrix contributes a factor $J_i^{j-2^k}$
to the determinant. We can remove the final $j-2^k$ rows and the columns of the anti-diagonal
matrix, after which we are left with a $2^k\times 2^k$ matrix. Let's denote it by $R$.
It consists of
the first $2^{k+1}-j$ columns of $A$ and the first $j-2^k$ columns of $A+B$,
which, as we noted above, is equal to $M_{i+1,k}(2^k)$. Another equality is
$A+B=2A+(-1)^{k+1}J_iD_k$.
So our $2^k\times 2^k$ matrix $R$ consists of a block from $A$ and a block from $2A+(-1)^{k+1}J_iD_k$.
By our conditions on $j$ (and this is the first place in the proof where we use this), the second block has as least as many columns as the first.
Perform a Gaussian elimination in which every column in the second
block is divided by two and subtracted from the corresponding column in
the first block.  This reduces $R$ to a matrix
\[
\left[
\begin{matrix}
0&N_1\\
(-1)^{k}\frac{J_i}2 D_{2^{k+1}-j}&N_2
\end{matrix}
\right]
\]
The upper right block $N_1$ corresponds to the $(j-2^k)\times (j-2^k)$
proper submatrix of $A+B$, which as we have seen above, is equal to
$M_{i+1,k}(j-2^k)$. Here we need that $j>2^k+2^{k+1}$.  The lower left
block contributes a factor $\left(\frac {J_i}{2}\right)^{2^{k+1}-j}$.
Ignoring the signs for the moment, we have reduced the matrix to
$M_{i+1,k}(j-2^k)$ and have obtained a factor
$\frac{J_i^{2^k}}{2^{2^{k+1}-j}}$, as required.

Now we still need to consider the sign.
We found $2^k$ factors in total,
the first $2^{k+1}-j$ were $J_i$ and the remaining $j-2^k$ were $J_i/2$.
The first came with a sign $(-1)^{k+1}$ and the latter with a sign
$(-1)^k$, which together produce the sign $(-1)^j$.
The determinant of an anti-diagonal $D_m$ is equal to $1$ if
$m=0$ or $1\text{ mod }4$ and $-1$ otherwise.
We encountered both $D_{j-2^k}$ and $D_{2^{k+1}-j}$.
If $k>1$, then this produces the sign $(-1)^j$,
but if $k=1$, it produces $(-1)^{j-1}$.
Finally, we need to observe the position of these two anti-diagonal
matrices as blocks in a matrix. Using that a matrix
$\left[\begin{matrix}0& S \\T&U\end{matrix}\right]$ with $s\times s$
block $S$ and $t\times t$ block $T$ has determinant
$(-1)^{st}\det(S)\det(T)$, this produces a factor $(-1)^{(j-2^k)^2}$ for
the first anti-diagonal matrix and a factor $(-1)^{(2^{k+1}-j)(j-2^k)}$
for the second.  Together this produces the sign $+1$.  If we consider
all three factors that we found, then we see that they produce $+1$ if
$k>1$ but $-1$ if $k=1$, which is $\epsilon_k$.
\end{proof}

For $\Delta_{\overline{\bf d}}(j)$ the computations are the same, but
we need to change some signs.  Whenever there is a $(-1)^{k+1}$ in the
computation above, it now becomes a $(-1)^k$, and vice versa.  The net
result is that recursion one applies to $\overline{\bf d}$ as well.

\subsubsection{Recursion two}

Our second recursion deals with $j$ that
have second significant digit zero in their binary expansion.

\begin{lemma}[Recursion two]
If $2^k< j \leq 2^k+2^{k-1}$ then
\[\Delta_{i,k+1}(j)=(-1)^j2^{(i-1)(2^{k+1}-j)}\cdot J_i^{2j-2^{k+1}}\cdot \Delta_{1,k}(2^{k+1}-j)\]
\end{lemma}

Observe that the recursion reduces $j$ by $2j-2^{k+1}$, which is equal
to the power of the Jacobsthal number. Also observe that recursions one
and two both apply to $j=2^k+2^{k-1}$,
and we obtain the equality
$$J_i^{2^k}\Delta_{i,k-1}(2^{k-1})=\frac{J_i^{2^k}}{2^{k-1}}\Delta_{i+1,k-1}{(2^{k-1})}.$$

\begin{proof}
By the same argument as in the proof of the first recursion, we end up
with the $2^k\times 2^k$ matrix $R$, only now the $A$ block is at least
as large as the $2A+(-1)^{k+1}J_iD_{2^k}$ block.  This time, we can use
the first block to reduce the second. Subtract every column of $A$
twice from the corresponding column in the second block.  We end up
with the matrix
\[
\left[
\begin{matrix}
N_1&0\\
N_2&(-1)^{k+1}J_i D_{j-2^k}
\end{matrix}
\right]
\]
The upper left block $N_1$ is a $(2^{k+1}-j)\times(2^{k+1}-j)$
submatrix of $A$, which is shorthand notation for $2^{i-1}P_k$, so
$N_1$ is in fact equal to $2^{i-1}M_{1,k}(2^{k+1}-j)$. Remembering that
we already encountered a determinant of $(-1)^{k+1}J_iD_{j-2^k}$ in the
reduction, with an extra sign $(-1)^{(j-2^k)^2}=(-1)^j$, it follows
that
\[\Delta_{i,k+1}(j)=(-1)^j2^{(i-1)(2^{k+1}-j)}\Delta_{1,k}(2^{k+1}-j)\cdot\det\left((-1)^{k+1}J_i
D_{j-2^k}\right)^2 ,\]
which reduces to the required recursion.
\end{proof}

Again, the computations are the same for $\overline\bf d$, except for the
final equation.  There the sign $(-1)^{k+1}$ changes to $(-1)^k$, which does
not affect the outcome.

\subsubsection{Applying the two recursions}

The two recursions combine to reduce any $2^k< j \leq 2^{k+1}$ to a $2^{k-1}\leq j'\leq 2^k$.
Each recursion decreases the index $k$ in $\Delta_{i,k}(j)$ by one.
If we start with an odd $j$ that has a binary expansion of length $k$, then after $k-2$ applications of the recursions,
we end up at $\Delta_{i,2}(3)$ for some $i$ (ignoring the additional factors that we picked up during the recursion).
Then we need to apply recursion two and end at $\Delta_{1,1}(1)$, ignoring the power of two. For $\bf d$ this is
equal to $1$, or $J_1$, and for $\overline{\bf d}$ this is equal to zero, or $J_0$, which explains why $\Delta_{\overline{\bf d}}(j)=0$
for odd $j$.
It follows that if $j$ is odd, then $\Delta_{i,k}(j)$ is a product of $2$'s and Jacobsthal numbers.
If we start with even $j$ then we end at $\Delta_{i,1}(2)$ after $k-1$ applications of the recursions.
Now for $\bf d$ we have that $\Delta_{i,1}(2)$ is equal to
\[
\left|
\begin{matrix}
J_i+J_{i-1}&J_{i-1}\\
J_{i-1}&J_i+J_{i-1}
\end{matrix}
\right|=
J_{i}(J_i+2J_{i-1})=J_iJ_{i+1},
\]
while for $\overline{\bf d}$ it is equal to
\[
\left|
\begin{matrix}
0&J_{i}\\
J_{i}&0
\end{matrix}
\right|=
-J_i^2,
\]
and so the quotient of $\Delta_{\bf d}(j)$ and $\Delta_{\overline{\bf d}}(j)$ is $-J_{i+1}/J_i$
for even $j$. Therefore, we restrict our attention to $\bf d$, because the corresponding result for 
$\overline{\bf d}$ is straightforward. 

If the recursion ends at $\Delta_{i,1}(2)$, then it produces two more
Jacobsthal numbers. If it ends at $\Delta_{1,1}(1)$ it produces $1$, or
$J_1$. It follows that the powers of the Jacobsthal numbers in
$\Delta_{\bf d}(j)$ add up to $j$. Of course, some powers may be
trivial since $J_0=0$ and $J_1=J_2=1$. We are now ready to prove
Theorem~\ref{main2}, which we restate as follows.

\begin{theorem}
$\Delta_{\bf d}(j)$ is a product of powers of Jacobsthal numbers $J_i^{n_i}$. The exponent
$n_i$ decreases as the index $i$ increases, with the exception of the largest non-trivial
power $J_{i+1}^{n_{i+1}}$ for which it may be true that $n_{i+1}=n_i=1$.
The sign of $\Delta_{\bf d}(j)$ depends on $j\text{ mod }4$.
It is negative if and only if $j=2\text{ mod }4$ or $3\text{ mod }4$.
\end{theorem}

\begin{proof}
We start with the sign first. It is true for $j\leq 4$ by direct
inspection. So we may assume that $k>1$ in the recursion and argue by
induction.  Recursion two reduces $j$ to $j-2^k$, which is equal modulo
4, without changing the sign. Recursion reduces $j$ to $2^{k+1}-j$, so
modulo 4 it interchanges $1$ and $3$. It also changes the sign in this
case, as it should, which finishes the induction.

The recursion produces powers of Jacobsthal numbers and perhaps powers
of two. But we need not compute the exponent of $2$ in $\Delta_{\bf
d}(j)$, since there are none \cite{Allouche&Peyriere&Wen&Wen:1998}.
Recursion one produces $J_i^{2^k}$ and increases the index $i$ by $1$.
Recursion two produces $J_i^{2j-2^{k+1}}$ and resets the index $i$ to
$1$. The exponent $2j-2^{k+1}$ is at most equal to $2^k$, so recursion
one produces the highest power of the two. This exponent decreases
(strictly) with $k$ and it immediately follows that the $n_i$ decrease
with $i$.  The only exception is that
in the final step of the iteration, when we end with
$\Delta_{i,1}(2)=J_iJ_{i+1}$, we obtain two additional Jacobsthal
factors.

\end{proof}

The following recursive formula was conjectured by Jason Bell and Kevin
Hare on November 26 2015, and independently by Tewodros Amdeberhan and
Victor Moll on December 6 2015:

\begin{theorem}
For odd $j$ we have
$$
\Delta_{\bf d}(2^m\cdot j)=\Delta_{\bf d}(j)^{2^m}\cdot \Delta_{\bf d}(2^m).
$$
\end{theorem}

\begin{proof}
By induction. If recursion one applies to $2^m\cdot j$, then it gives
\[\Delta_{i,k+m+1}(2^m\cdot j)=\frac{J_i^{2^{k+m}}}{2^{2^{k+m+1}-j}}\cdot\Delta_{i+1,k+m}(2^m(j-2^{k})) ,\]
which in particular produces a Jacobsthal power $J_i^{2^{k+m}}$ and reduces $2^m\cdot j$ to
$2^m(j-2^k)$.
If recursion one applies to $2^m\cdot j$, then it also does to $j$,
and it produces the Jacobsthal
power $J_i^{2^k+m}$ and reduces $j$ to $(j-2^k)$.
Similarly, if recursion two applies, 
then it produces a Jacobsthal power $J_i^{2^m(2j-2^{k+1})}$ and
reduces $2^m\cdot j$ to $2^m(2^{k+1}-j)$,
while it produces a Jacobsthal power $J_i^{(2j-2^{k+1})}$
for $j$. We can ignore the powers of two, as before, and it is not hard to check that the signs are
equal on both sides of the equation, so we may ignore that as well.
The recursion for odd $j$
ends at $\Delta_{i,1}(1)$,
while for $2^m\cdot j$ it reaches $\Delta_{1,m}(2^m)$ at that
point, and we conclude that the recursive formula holds.
\end{proof}

\section{Hankel determinants of the shifted sequence}

The Hankel matrix of the shifted sequence $s_qs_{q+1}s_{q+2}\cdots$
is given by
\begin{equation}
M_{{\bf s},q} (j) = \left[
	\begin{array}{cccc}
	s_q & s_{q+1} & \cdots & s_{q+j-1} \\
	s_{q+1} & s_{q+2} & \cdots & s_{q+j} \\
	\vdots & \vdots & \ddots & \vdots \\
	s_{q+j-1} & s_{q+j} & \cdots & s_{q+2j-2}
	\end{array} \right] .
\end{equation}
and the corresponding Hankel determinant is $\Delta_{{\bf s},q}(j)$. 
Observe that $M_{{\bf s},q} (j)$ occurs as a $j\times j$ submatrix
in $M_{\bf s} (q+j)$.
For $q>0$ the Hankel determinants of $\bf d$ are no longer products of Jacobsthal numbers, but
the first few terms indicate some interesting patterns: 
\begin{table}[H]
\small{
\begin{center}
\begin{tabular}{c|rrrrrrrrrrrrrrrr}
 \backslashbox{$q$}{$j$}&1&2&3&4&5&6&7&8&9&10&11&12&13&14&15&16\\
\hline
0&1& 1&$-1$&$-3$& 1& 1&$-1$&$-15$& 1& 1&$-1$&$-3$&  1&  1& $-9$ & $-495$\\
1&0&$-1$& 1& 2& 1& 1&$-4$&11& 3&$-2$& 3&$-3$& $-2$& $-7$&141&354\\
2&1& 0&$-1$&$-1$& 0& 1&$-1$&$-8$& 1& 3& 0&$-3$& 1& 40& $-9$ & $-253$\\
3&1&$-1$&$-1$& 0& 0& 1&$-3$& 5&$-3$& 0& 0&$-3$&  $-11$& $-7$&$-17$&180\\
4&1& 1& 0& 0& 0& 1&$-1$&$-3$& 0& 0& 0&$-3$&1&  1& $-4$ & $-128$\\
5&0&$-1$& 0& 0& 1&$-2$& 3& 0& 0& 0& 0&$-3$&  $-2$& $-5$& 52& 76\\
6&1& 1& 0&$-1$& 1& 3& 0& 0& 0& 0& 0&$-3$&1& 21& $-4$&$-45$\\
7&0&$-1$&$-1$& 2&$-3$& 0& 0& 0& 0& 0& 0&$-3$&$-8$& $-5$& $-7$& 26\\
8&1& 1&$-1$&$-3$& 0& 0& 0& 0& 0& 0& 0&$-3$&1&  1& $-1$&$-15$\\
9&0&$-1$& 1& 3& 0& 0& 0& 0& 0& 0& 3& 5 & 3&  4&$-15$&  0\\
10&1& 0&$-1$&$-3$& 0& 0& 0& 0& 0& 3&$-1$&$-8$& 1& 15&  0&  0\\
11&1&$-1$&$-2$& 3& 0& 0& 0& 0&$-3$&$-2$&$-3$&11&15&  0&  0&  0\\
12&1& 1&$-1$&$-3$& 0& 0& 0&$-3$& 1& 1&$-1$&$-15$&  0&  0&  0&  0\\
13&0&$-1$& 1& 3& 0& 0& 3& 5& 3& 4&$-15$& 0&  0&  0&  0&  0\\
14&1& 0&$-1$&$-3$& 0& 3&$-1$&$-8$& 1&15& 0& 0&  0&  0&  0&  0\\
15&1&$-1$&$-2$& 3&$-3$&$-2$&$-3$&11&15& 0& 0& 0&  0&  0&  0&  0\\
16&1& 1&$-1$&$-3$& 1& 1&$-1$&$-15$& 0& 0& 0& 0& 0&  0&  0&  0
\end{tabular}
\end{center}}
\end{table}

One pattern that emerges from this table is that the $q=2^k$'th row starts
with the first $2^{k-1}$ numbers of the first row, followed by
$2^{k-1}$ zeroes. This follows directly from our results.

\begin{proposition}
If $j\leq 2^{k-1}$ then $\Delta_{{\bf d},{2^{k}}}(j)=\Delta_{\bf d}(j)$, and if $2^{k-1}<j\leq 2^{k}$ then $\Delta_{{\bf d},{2^{k}}}(j)=0$.
\end{proposition}

\begin{proof}
The Hankel matrix $M_{{\bf d},2^k}(2^k)$ is the lower left $Q_k$ block of $M_{\bf d}(2^{k+1})=\left[\begin{matrix}P_k&Q_k\\Q_k&P_k\end{matrix}\right]$,
which, in turn, is equal to $Q_k=\left[\begin{matrix}P_{k-1}&P_{k-1}\\P_{k-1}&P_{k-1}\end{matrix}\right]$.
It immediately follows that $\Delta_{{\bf d},{2^{k}}}(j)=\Delta_{\bf d}(j)$ if $j\leq 2^{k-1}$, since this is the determinant
of the $j\times j$ block in $P_{k-1}$, and that $\Delta_{{\bf d},{2^{k}}}(j)=0$ for $2^{k-1}<j<2^k$ since row one
of the matrix is repeated in row $2^{k-1}+1$. 
\end{proof}

The table again indicates that something interesting is going on when
$j$ is a power of $2$. A full analysis is probably not that easy.
Allouche et al.~\cite{Allouche&Peyriere&Wen&Wen:1998} needed 16
recursions to resolve the Hankel table of the Thue-Morse sequence
modulo~2.

\section{Conclusion}

We set out to study the values of Hankel determinants of the Thue-Morse
sequence at powers of 2, and we ended up studying Hankel determinants
of the period-doubling sequence. The values of the Hankel determinants
$\Delta_{\bf t} (n)$ for the Thue-Morse sequence continue to be
mysterious.

\section{Acknowledgment}

We thank the referee for several helpful corrections.

\newcommand{\noopsort}[1]{} \newcommand{\singleletter}[1]{#1}

\end{document}